\documentclass[a4paper,12pt,]{article}
\usepackage{hyperref}
\usepackage{amsfonts}
\usepackage{amsmath}
\usepackage{amssymb}
\usepackage{amsthm}
\usepackage{enumerate}
\usepackage[mathscr]{eucal}
\usepackage{eqlist}
\usepackage{graphicx}
\usepackage{esint}
\usepackage{color}

\usepackage[body={15.6true cm, 24.9true cm}]{geometry}

\newtheorem{Theorem}{Theorem}[section]
\newtheorem{Definition}{Definition}
\newtheorem{Lemma}[Theorem]{Lemma}

\newtheorem{Remark}[Theorem]{Remark}

\numberwithin{equation}{section} \allowdisplaybreaks
\usepackage{mathdots}

\renewcommand\abstract{{\bf Abstract}}

\begin{document}
	\title{Uniqueness of the critical points of solutions to two kinds of semilinear elliptic equations in higher dimensional domains\footnote{\footnotesize The work is supported by National Natural Science Foundation of China (No.12001276, No.12271093, No.12326303, No.12326301).}}

	\author{Haiyun Deng$^{1}$, Jingwen Ji$^{2}$, Feida Jiang$^{2}$\footnote{\footnotesize Corresponding author, E-mail: hydeng@nau.edu.cn(H. Deng), jijingwen98@163.com(J. Ji), jiangfeida@seu.edu.cn(F. Jiang), jiabinyin@126.com(J. Yin)}, Jiabin Yin$^{3}$\\[12pt]
		\emph {\scriptsize $^{1}$Department of Applied Mathematics, Nanjing Audit University, Nanjing, 211815, China;}\\
		\emph {\scriptsize $^{2}$School of Mathematics and Shing-Tung Yau Center of Southeast University, Southeast University, Nanjing, 211189, China;}\\
		\emph {\scriptsize $^{3}$School of Mathematics and Statistics, Guangxi Normal University, Guilin, 541004, China}}
	
	\date{}
	\maketitle

	\renewcommand{\labelenumi}{[\arabic{enumi}]}
	
	\begin{abstract}{\bf:}{\footnotesize
			~In this paper, we provide an affirmative answer to the {\it conjecture A} for bounded simple rotationally symmetric domains $\Omega\subset \mathbb{R}^n(n\geq 3)$ along $x_n$ axis. Precisely, we use a new simple argument to study the symmetry of positive stable solutions for two kinds of semilinear elliptic equations. To do this, when $f(\cdot,s)$ is convex with respect to $s$, we show that the positivity of the first eigenvalue of the corresponding linearized operator in somehow symmetric domains is a sufficient condition for the symmetry of $u$.  Moreover, we prove the uniqueness of critical points of a positive stable solution to semilinear elliptic equation $-\triangle u=f(\cdot,u)$ with zero Dirichlet boundary condition for simple rotationally symmetric domains in $\mathbb{R}^n$ by continuity method and a variety of maximum principles.}
		
	\end{abstract}
	
	
	
	\section{Introduction and main results}
	~~~~In this article, we mainly consider the uniqueness of critical points of stable solutions to the following semilinear elliptic equation with Dirichlet boundary condition
	\begin{equation}\label{1.1}\begin{array}{l}
			\left\{
			\begin{array}{l}
				-\triangle u=f(\cdot,u)~~\mbox{in}~~\Omega,\\
				u=0 ~~\mbox{on}~~\partial\Omega,
			\end{array}
			\right.
	\end{array}\end{equation}
	where $\Omega$ is a bounded domain in $\mathbb{R}^{n}(n\geq 3)$, and $f(x,u)\in C^{1,1}(\Omega,\mathbb{R})$ is a convex function with respect to $u$.

	Critical points of solutions to elliptic equations is an important research topic. The type and computation of the number of critical points of positive solution $u$ to Poisson equation $\triangle u+f(u)=0$ is a classic and fascinating problem. However, there are few conclusions about the geometric distribution of critical points of solutions in higher dimensional domains. In this regard, there is a following well known seemingly simple but challenging conjecture, relevant statements can be found in \cite{CabreChanillo} page 2, and \cite{Kawohl} page 103.
	
	\vspace{0.2cm}
	\noindent {\bf Conjecture A.} {\it Let $u$ be a positive solution of the following boundary value problem
		\begin{equation*}\label{}\begin{array}{l}
				\left\{
				\begin{array}{l}
					-\triangle u=f(u)~~\mbox{in}~~\Omega\subset \mathbb{R}^n,~~n\geq 3,\\
					u=0 ~~\mbox{on}~~\partial\Omega,
				\end{array}
				\right.
		\end{array}\end{equation*}
		where $f$ is a  monotone function and $\Omega$ is a general strictly convex domain in $\mathbb{R}^n (n\geq 3)$, then $u$ has a unique non-degenerate critical point in $\Omega$.}
	\vspace{0.2cm}
	
	In higher dimensional space, as far as we know, there are few works about the geometric distribution of critical points of solutions to elliptic equations. Under the assumption of the existence of a semi-stable solution of Poisson equation $-\triangle u=f(u)$ with $u| _{\partial\Omega}=0$, Cabr\'{e} and Chanillo \cite{CabreChanillo} proved that the positive solution $u$ has a unique non-degenerate critical point in bounded smooth convex domains of revolution with respect to an axis in $\mathbb{R}^{n}(n\geq3)$. In 2017, Alberti, Bal and Di Cristo \cite{Alberti} investigated the existence of critical points for solutions to second-order elliptic equations of the form $\nabla\cdot\sigma(x)\nabla u=0$ over a bounded domain with prescribed boundary conditions in $\mathbb{R}^{n}(n\geq3).$ In 2019, Deng, Liu and Tian \cite{Deng2019} considered the geometric distribution of critical points of solution to mean curvature equations with Dirichlet boundary condition on some special symmetric domains in $\mathbb{R}^{n}(n\geq3)$ by the projection method. In 2023, Grossi and Luo \cite{GrossiLuo} studied the number and location of critical points of positive solutions of semilinear elliptic equations  $\triangle u+f(u)=0$ with $u_{\partial\Omega_\varepsilon}=0$ in domains $\Omega_\varepsilon=\Omega\setminus B(P,\varepsilon)$, where $\Omega\subset\mathbb{R}^n(n\geq 3)$ is a smooth bounded domain. In fact, they mainly settled the relationship between the number and position of critical points of solutions in domain $\Omega_\varepsilon$ and domain $\Omega$. Other related results can be found in \cite{AlessandriniMagnanini1,Chen,De,Deng2018,Deng2022,Deng2023,GrossiIanni,Magnanini} and references therein.
	
	There are a few more results about the Hausdorff measure estimation of critical point and singular sets of solutions to elliptic equations in higher dimensional space. In 1996, M. Hoffmann-Ostenhof, T. Hoffmann-Ostenhof and N. Nadirashvili \cite{Hoffmann-Ostenhof} proved that the singular sets of smooth solutions for the second order elliptic equation in three dimensional space have locally finite one dimensional Hausdorff measure. This result was generalized to higher dimensional space in 1999 (see \cite{Hardt1}). In 1998, Han, Hardt and Lin \cite{Han1998} concerned with the geometric measure of singular sets of weak solutions of elliptic second-order equations. They gave a uniform estimate on the measure of singular sets in terms of the frequency of solutions. In 2003, Han, Hardt and Lin \cite{Han2003} studied the singular sets of solutions to arbitrary order elliptic equations. In 2015, Cheeger, Naber and Valtorta \cite{Cheeger} introduced some techniques for giving improved estimates of the critical set, including the Minkowski type estimates on the effective critical set $C_r(u)$, which roughly consists of points $x$ such that the gradient of $u$ is small somewhere on $B_r(x)$ compared to the nonconstancy of $u.$ In 2017, Naber and Valtorta \cite{Naber} introduced some new techniques for estimating the critical set and singular set, which avoids the need of any such $\varepsilon$-regularity lemma. In 2020, Deng, Liu and Tian \cite{Deng2020} gave a new classification of singular sets of solutions $u$ to a kind of linear elliptic equations. They obtained the Hausdorff measure estimation of singular sets of solutions $u$ by defining the $j-$symmetric function and the $j-$symmetric singular set of $u$. See \cite{Lin2024} and references therein for some related work.
	
	It is well known that the method of level set is generally used to consider the geometric distribution of critical points of solutions in the planar domains, through the fine analysis of level sets. However, this method is very difficult to be generalized to the higher dimensional domains. In this paper, we try to build the continuity method based on a variety of maximum principles, which is natural and effective to settle such a problem in higher dimensional space. Continuity method has been widely used in various problems effectively and has a certain degree of recognition. The typical example is to establish Schauder estimates for elliptic equations; and methods of moving plane introduced by Serrin \cite{Serrin}, developed by Gidas, Ni and Nirenberg \cite{GidasNi}, Berestycki and Nirenberg \cite{Berestycki} to prove symmetry of positive solution of Dirichlet boundary value problems. For the sake of clarity, we now explain the method and our ideas. The main idea our proof depends on the stability of Morse function and a variety of maximum principles. The main finding of this paper is that the solution $u$ of (\ref{1.1}) is stable under small domain deformation, namely, the number, location and type of critical points remain preserved.  This is a crucial finding, which is quite surprising. This is an essentially different method from the previous projection method in \cite{CabreChanillo,Deng2019}. Essentially, the projection method requires that the domains be strictly convex. We consider the distribution of critical points of stable solutions in a kind of special symmetric domain - simple rotationally symmetric domain along $x_n$ axis.
	
	Now we introduce some definitions.
\begin{Definition}\label{def1}
		(stable solution) We say that $u$ is a stable solution of (\ref{1.1}) if the linearized operator $\mathfrak{L}=-\triangle -f'(\cdot,u)$ of (\ref{1.1}) at $u$ is definite, that is, if
\begin{equation}\label{1.2}
		 (\mathfrak{L}\phi, \phi)=\int_\Omega|\nabla\phi|^2dx-\int_\Omega f^\prime(\cdot,u)\phi^2dx>0,
	\end{equation}
for any $\phi\in H^1_0(\Omega)$, where $f^\prime(\cdot,u)$ denotes the derivative of $f(\cdot,u)$ with respect to $u$. Equivalently, $u$ is a stable solution if the first eigenvalue $\lambda_1(\mathfrak{L},\Omega)$ of the operator $\mathfrak{L}$ in $\Omega$ is positive.
 \end{Definition}
	
	\begin{Definition}\label{def2}
		($x_i$-convex) We say that a domain $\Omega$ is $x_i$-convex, if for any two points $P,Q\in\Omega$ satisfy that, the line segment $\overline{PQ}$ is parallel to the $x_i$ axis, then $\overline{PQ}\subset \Omega.$
 \end{Definition}
	
	\begin{Definition}\label{def3}
 We say that a domain $\Omega$ is simple rotationally symmetric along $x_n$ axis if\\
		(1) it is symmetric with respect to the plane $x_n=0$;\\
		(2) it is invariant under any action of orthogonal group $A\in O(n)$, where $A$ fixes the $x_n$ axis;\\
		(3) it is $x_i$-convex for any $i=1,2,\cdots,n$.
\end{Definition}
	
	By this definition,  we know such domains satisfy that there exists a continuous monotone nonincreasing function $g:\mathbb{R}\rightarrow \mathbb{R}^+$ with $R$ being its first zero, that is, $g(r)>0, 0\leq r<R; g(0)=a, g(R)=0$ such that
	$$\Omega\subset~\mbox{the cylinder of}~\{x=(r,x_n): -a<x_n<a, 0\leq r:=(\sum_{i=1}^{n-1}x_i^2)^{\frac{1}{2}}<R\}.$$
	It means that such domains can be parameterized as
	\begin{equation*}
		\Omega=\{x=(r,x_n):-g(r)<x_n<g(r), 0\leq r=(\sum_{i=1}^{n-1}x_i^2)^{\frac{1}{2}}<R\},
	\end{equation*}
	where $a$ is the maximal semi axis of all domains $\Omega$ in $x_n$ direction. We should mention that for $r<R$ but close to $R$, it is often convenient to express the boundary $\partial\Omega$ as $r=h(|x_n|)$, where $h$ is the inverse of $g$. We say that $\Omega$ is strictly simple if $g^\prime(r)<0$ for any $r\in(0,R)$.

	The main goal of this article is to investigate the distribution of critical points of stable solutions to two kinds of semilinear elliptic equations with Dirichlet boundary condition in higher dimensional domains. The main results of this article are as follows:
	
	\begin{Theorem}\label{thm1.1}
		Let $\Omega\subset \mathbb{R}^n(n\geq 3)$ be a bounded simple rotationally symmetric domain along $x_n$ axis. Let $u$ be a positive stable solution to the following equation
		\begin{equation}\label{1.3}\begin{array}{l}
				\left\{
				\begin{array}{l}
					-\triangle u=f(x,u)~~\mbox{in}~~\Omega,\\
					u=0 ~~\mbox{on}~~\partial\Omega,
				\end{array}
				\right.
		\end{array}\end{equation}
		where $f(x,u)\in C^{1,1}(\Omega,\mathbb{R})$ is a convex function with respect to $u$, $f(x^\prime,x_n,u)=f(x^\prime,-x_n,u)$, $f(x^\prime,x_n,u)=f(|x^\prime|,x_n,u)$, $x^\prime=(x_1,\cdots,x_{n-1})$ and $f$ is decreasing in the $x_i$-variable for $x_i>0(i=1,\cdots, n)$. Then   \\
		(1) $u$ is rotationally symmetric with respect to $x_n$ axis;\\
		(2) $u_{x_n}<0,$ for $x_n>0$ ;\\
		(3) $x_i u_{x_i}<0,$ for $x_i\neq 0,i=1,2,\cdots,n-1$;\\
		(4) $\frac{\partial u(r, x_n)}{\partial r}<0,$ for $r\neq 0$, where $r=|x'|, x'=(x_1,x_2,\cdots,x_{n-1})$;\\
		(5) $u$ has a unique critical point $p$ in $\Omega$ and $p$ is a non-degenerate maximum point.
	\end{Theorem}

	\begin{Theorem}\label{thm1.2}
		Let $\Omega\subset \mathbb{R}^n(n\geq 3)$ be a bounded simple rotationally symmetric domain along $x_n$ axis. Let $u$ be a positive stable solution to the following semilinear elliptic equation
		\begin{equation}\label{1.4}\begin{array}{l}
				\left\{
				\begin{array}{l}
					-\triangle u=f(u)~~\mbox{in}~~\Omega,\\
					u=0 ~~\mbox{on}~~\partial\Omega,
				\end{array}
				\right.
		\end{array}\end{equation}
		where $f(u)\in C^{1}(\mathbb{R})$ is a convex function. Then   \\
		(1) $u$ is rotationally symmetric with respect to $x_n$ axis;\\
		(2) $u_{x_n}<0,$ for $x_n>0$ ;\\
		(3) $x_i u_{x_i}<0,$ for $x_i\neq 0,i=1,2,\cdots,n-1$;\\
		(4) $\frac{\partial u(r, x_n)}{\partial r}<0,$ for $r\neq 0$, where $r=|x'|, x'=(x_1,x_2,\cdots,x_{n-1})$.
	\end{Theorem}
	
	According to the results of Theorem \ref{thm1.2} and the maximum principle, we can easily obtain the following theorem, which provides an affirmative answer to the {\it conjecture A} in a bounded simple rotationally symmetric domain $\Omega\subset \mathbb{R}^n(n\geq 3)$ along $x_n$.
	
	\begin{Theorem}\label{thm1.3}
		Suppose that $\Omega$ is a bounded simple rotationally symmetric domain along $x_n$ axis in $\mathbb{R}^n(n\geq 3)$. Let $u$ be a positive stable solution of (\ref{1.4}), where $f(u)\in C^{1}(\mathbb{R})$ is a convex function. Then $u$ has a unique critical point in $\Omega$. Moreover, this unique critical point is the maximum point of $u$, and it is non-degenerate.
	\end{Theorem}
	
	\begin{Remark}\label{rem1.4}
By the assumption of function $f(\cdot,s)$, the following two classical cases which fit in this class:\\
i) The case of torsion problem $f(\cdot,s)=1.$\\
ii) The case of linear function $f(\cdot,s)=\widehat{\lambda}s$, where $\widehat{\lambda}<\lambda_1$ and $\lambda_1$ is the first Dirichlet eigenvalue of operator $-\Delta$ in $\Omega$.
	\end{Remark}

	\begin{Remark}\label{rem1.5}
		Notice that if $f(\cdot,u)\geq 0$ for all $u$ in (\ref{1.1}), by the strong maximum principle, then it can be guaranteed that the non-trivial solution must be positive solution.
	\end{Remark}

	For the sake of clarity, we will illustrate the main idea of the proof of Theorem {\ref{thm1.1}} and {\ref{thm1.2}}. Firstly, we consider a special symmetric domain - simple rotationally symmetric domain along $x_n$ axis, and use the positivity of the first eigenvalue of the corresponding linearized operator to study the symmetry of a positive solution. Secondly, we construct a family of rotationally symmetric domains $\Omega_t (t\in[0,1])$ in $\mathbb{R}^n$, where $\Omega_0=B_a$ is the initial domain with radius $a$ and $\Omega_1=\Omega$ is our target domain. Each solution $u$ of (\ref{1.1}) satisfies the all conclusions of Theorem {\ref{thm1.1}} and {\ref{thm1.2}} in initial domain $\Omega_0$. In the end, we will show that $u$ is a Morse function which has still exactly one critical point as the domains deform by continuity method and a variety of maximum principles. It means that $u$ is stable in $\Omega_t$, which is a crucial property. The advantages of our method are as follows:\\
	\emph{(1) It can deal with a larger class of higher dimensional domains.\\
		(2) The domain $\Omega$ that we consider simply requires to be $x_i$-convex and not convex.\\
		(3) It can be extended to more general uniformly elliptic equations, e.g., quasilinear uniformly elliptic equations.}
	
	The rest of this article is written as follows. In Section 2, we introduce Serrin's comparison principle, BV theorem, Serrin's Lemma - a generalized version of Hopf boundary lemma at corners of boundary, and other basic result needed in the proof of main results. Among others,  we will prove Theorem \ref{thm1.1} in Section 3, Theorems \ref{thm1.2} and \ref{thm1.3} in Section 4.

	\section{Preliminaries}
	~~~~In this section, firstly we will recall some facts about maximum principle. Suppose that $\Omega$ is a bounded domain in $\mathbb{R}^{n}.$ Let $L$ be a uniformly second order elliptic operator, be defined as
	\begin{equation}\label{2.1}
		Lu=\sum_{i,j=1}^n a_{ij}(x)u_{x_ix_j}+\sum_{i=1}^n b_{i}(x)u_{x_i}+c(x)u,
	\end{equation}
	where $a_{ij}=a_{ji}, a_{ij}\in C(\overline{\Omega}), b_i,c\in L^{\infty}(\Omega),$ and $\sum_{i,j=1}^n a_{ij}(x)\xi_i\xi_j\geq \lambda|\xi|^2$ for any $x\in \Omega, \xi\in \mathbb{R}^n$ and some positive constant $\lambda.$ Now let $\lambda_1(L,\Omega)$ be the first eigenvalue of the above elliptic operator $L$ in $\Omega$, given by
	$$\lambda_1(L,\Omega)=\sup \{\lambda:\exists u>0~\mbox{such~that}~(L+\lambda)u\leq 0\},$$
	where $u\in W^{2,n}_{loc}\cap C(\overline{\Omega}).$
	
	We will need the following generalized versions of maximum principle, BV theorem, and Hopf lemma at the corners of boundary.
	\begin{Lemma}\label{le2.1}
		(Serrin's comparison principle, see \cite{Serrin}) Assume that $u\in C^2(\Omega)\cap C(\overline{\Omega})$ satisfies $Lu\geq 0,$ where operator $L$ is defined as (\ref{2.1}). If $u\leq 0$ in $\Omega$, then either $u<0$ in $\Omega$ or $u\equiv 0$ in $\Omega$.
	\end{Lemma}
	
	\begin{Lemma}\label{le2.2}
		(BV Theorem, see \cite{BerestyckiVaradhan}) The maximum principle holds in $\Omega$ if and only if $\lambda_1(L,\Omega)>0.$
	\end{Lemma}
	
	\begin{Lemma}\label{le2.3}
		(Serrin's Lemma, see \cite[Lemma S]{GidasNi} and \cite[Lemma 2]{Serrin})
		Let $\Omega$ be a domain in $\mathbb{R}^n$ with the origin $o$ on its boundary. Suppose that near $o$ the boundary consists of two transversally intersecting $C^2$ hypersurfaces $\varrho=0$ and $\sigma=0.$ Assume that $\varrho, \sigma<0$ in $\Omega.$ Let $u$ be a function in $C^2(\overline{\Omega})$, with $u<0$ in $\Omega, u(o)=0$, satisfying the differential inequality $Lu\geq 0$, where operator $L$ is defined as (\ref{2.1}). Suppose that
		\begin{equation}\label{2.2}
			\sum_{i,j=1}^n a_{ij}(x)\varrho_{x_i}\sigma_{x_j}\geq 0~\mbox{at}~o.
		\end{equation}
		If this is zero, then assume furthermore that $a_{ij}\in C^2$ in $\overline{\Omega}$ near $o$ with
		\begin{equation}\label{2.3}
			D(\sum_{i,j=1}^n a_{ij}(x)\varrho_{x_i}\sigma_{x_j})=0~\mbox{at}~o,
		\end{equation}
		for any first order derivative $D$ at $o$ tangent to the submanifold $\{\varrho=0\}\cap\{\sigma=0\}.$ Then, for any direction $s$ at $o$ which enters $\Omega$ transversally to each hypersurface,
		\begin{equation}\label{2.4}\begin{array}{l}
				\left\{
				\begin{array}{l}
					\partial_su<0~\mbox{at}~o~\mbox{in case of strict inequality in (\ref{2.2})},\\
					\partial_su<0~\mbox{or}~\partial_{ss}u<0~\mbox{at}~o~\mbox{in case of equality in (\ref{2.2})}.
				\end{array}
				\right.
		\end{array}\end{equation}
	\end{Lemma}
	
	Note that conditions (\ref{2.2}) and (\ref{2.3}) are invariant under change of coordinates and the choices of the particular functions $\varrho$ and $\sigma$ representing the bounding hypersurfaces.
	
	To make it easier to understand, we need to give the following definition.
	
	\begin{Definition}\label{def4}
(Morse function) We say that a function $u$ is Morse function, if all the critical points of $u$ are non-degenerate, namely, $\det|\partial_{ij}u|\neq 0$ at all its critical points.
   \end{Definition}

	The main difference of our method is that we find that the stability of critical points for Morse function $f$, which ensures that the sign of derivative with respect to $x_i(i=1,2,\cdots,n)$ in the corresponding domain remains the same.
	Therefore we need the following well known Morse theorem results. In such smooth maps setting, stability means that a small perturbation of $f$ keeps the total number, location and type of the critical points preserved. See \cite{Anosov,Milnor} for details.

	\begin{Lemma}\label{le2.4} (see \cite{Anosov,Milnor})\\
		(1) The stable smooth maps $f: M^n\rightarrow R$ from a manifold $M^n$ to real line form an everywhere dense
		set in the space of all smooth maps.\\
		(2) For a map $f$ to be stable it is necessary and sufficient that the following two conditions are satisfied: (a)
		All critical points of $f$ are nondegenerate; (b) All the critical values are distinct.\\
		(3) A map $f: M^n\rightarrow R$ is stable at a point $x_0$ if and only if there are neighborhoods of $x_0\in M^n$ that $f$ can be
		expressed in the following forms in a local coordinate
		\begin{equation}\label{2.4}\begin{array}{l}
				\left\{
				\begin{array}{l}
					f(x)=x_1,\\
					\mbox{or}\\
					f(x)=x_1^2+x_2^2+\cdots +x_k^2-x_{k+1}^2-\cdots -x_n^2.
				\end{array}
				\right.
		\end{array}\end{equation}
	\end{Lemma}

	\section{Proof of Theorem \ref{thm1.1} }
	~~~~In this section, we will use a simple argument to study the symmetry of a positive stable solution $u$ of (\ref{1.3}) in $\Omega$, this idea can be seen in \cite{Pacella}, for the sake of completeness, in our setting, we will give the following complete proof. Now we define the following hyperplane $T_n$, domains $\Omega_n^{-}$ and $\Omega_n^{+}$:
	\begin{equation}\label{3.1}\begin{array}{l}
			\left\{
			\begin{array}{l}
				T_n=\{x=(x_1,x_2,\cdots,x_n)\in\Omega,~x_n=0\},\\
				\Omega_n^{-}=\{x\in\Omega,~x_n<0\},\\
				\Omega_n^{+}=\{x\in\Omega,~x_n>0\},
			\end{array}
			\right.
	\end{array}\end{equation}
	where $\Omega$ is a simple rotationally symmetric domain along $x_n$ axis.
	
	Next, we will prove the symmetry of a positive stable solution to (\ref{1.3}) in $\Omega,$ where $\Omega$ is a simple rotationally symmetric domain along $x_n$ axis.  Let $u$ be a stable solution of (\ref{1.3}). We denote by $\mathfrak{L}$ and $\lambda_1(\mathfrak{L},D)$ the linearized operator at $u$ and the first eigenvalue of $\mathfrak{L}$ in a subdomain $D\subset\Omega$ respectively, namely
	$$\mathfrak{L}=-\triangle -f'(x,u),$$
	where $f'$ denotes the derivative of $f(x,u)$ with respect to $u.$
\begin{Lemma}\label{le3.1}
		Let $u$ be a positive stable solution of (\ref{1.3}). If $f(x,u)$ is a convex function with respect to $u$, then this solution $u$ is unique.
	\end{Lemma}
	\begin{proof} We set up the usual contradiction argument. Without loss generality, we suppose that $u$ and $v$ are two positive stable  solutions of (\ref{1.3}) in $\Omega\in\mathbb{R}^n(n\geq 3)$, then we have that
\begin{equation}\label{3.2}
\begin{split}
-\Delta u=f(x,u),
\end{split}
\end{equation}
and
\begin{equation}\label{3.3}
\begin{split}
-\Delta v=f(x,v).
\end{split}
\end{equation}

Assume that $w_1=u-v,$  by (\ref{3.2}) minus (\ref{3.3}), we obtain that
\begin{equation*}\label{}
			\left\{
			\begin{array}{l}
				-\Delta w_1=f(x,u)-f(x,v)\geq f^\prime(x,v)w_1~\mbox{in}~\Omega, \\
				w_1=0~\mbox{on}~\partial\Omega.
			\end{array}
			\right.
		\end{equation*}
That is
\begin{equation}\label{3.4}
			\left\{
			\begin{array}{l}
				\mathfrak{L}w_1:=-\Delta w_1-f^\prime(x,v)w_1\geq 0~\mbox{in}~\Omega, \\
				w_1=0~\mbox{on}~\partial\Omega.
			\end{array}
			\right.
		\end{equation}

In addition, denote by $w_2=v-u,$  by (\ref{3.3}) minus (\ref{3.2}), we can also have
\begin{equation*}\label{}
			\left\{
			\begin{array}{l}
				-\Delta w_2=f(x,v)-f(x,u)\geq f^\prime(x,u)w_2~\mbox{in}~\Omega, \\
				w_2=0~\mbox{on}~\partial\Omega.
			\end{array}
			\right.
		\end{equation*}
That is
\begin{equation}\label{3.5}
			\left\{
			\begin{array}{l}
				\mathfrak{L}w_2:=-\Delta w_2-f^\prime(x,u)w_2\geq 0~\mbox{in}~\Omega, \\
				w_2=0~\mbox{on}~\partial\Omega.
			\end{array}
			\right.
		\end{equation}
By the assumption of stable solution to (\ref{1.3})  and Definition 1, we know that the first eigenvalue $\lambda_1(\mathfrak{L},\Omega)$ of the operator $\mathfrak{L}$ in $\Omega$ is positive. Lemma \ref{le2.2} means that $w_1\geq 0$ and $w_2\geq 0$ in $\Omega$, namely, $u=v$ in $\Omega$.  This completes the proof of Lemma \ref{le3.1}.\end{proof}

	\begin{Lemma}\label{le3.2}
		Let $u$ be a stable solution of (\ref{1.3}). If $f(x,u)$ is convex with respect to $u$, then $u$ is symmetric with respect to the hyperplane $T_n$, i.e. $u(x_1,x_2,\cdots,-x_n)=u(x_1,x_2,\cdots,x_n)$.
	\end{Lemma}
	\begin{proof} Due to the hypothesis of stable solution to (\ref{1.3}), Definition 1 and the monotonicity of the first eigenvalues with respect to the domains, we have that $\lambda_1(\mathfrak{L},\Omega)>0$, $\lambda_1(\mathfrak{L},\Omega_n^{+})>0$ and $\lambda_1(\mathfrak{L},\Omega_n^{-})>0$. We suppose that $v^+$ and $v^-$ are the reflected functions of $u$ in the domains $\Omega_n^+$ and $\Omega_n^-$, respectively,
		\begin{equation*}\label{}
			\left\{
			\begin{array}{l}
				v^+(x)=u(x_1,x_2,\cdots,-x_n) ~~\mbox{for}~x\in\Omega_n^+,\\
				v^-(x)=u(x_1,x_2,\cdots,-x_n) ~~\mbox{for}~x\in\Omega_n^-.
			\end{array}
			\right.
		\end{equation*}
		Since $f(x,u)$ is convex with respect to $u$, we have
		\begin{equation*}\label{}
			\left\{
			\begin{array}{l}
				f(x,v^+(x))-f(x,u(x))\geq f'(x,u(x))(v^+(x)-u(x)) ~~\mbox{in}~x\in\Omega_n^+,\\
				f(x,v^-(x))-f(x,u(x))\geq f'(x,u(x))(v^-(x)-u(x))~~\mbox{in}~x\in\Omega_n^-.
			\end{array}
			\right.
		\end{equation*}
		 According to (\ref{1.3}), defining the functions $V^+=v^+-u$ and $V^-=v^--u$, then we have
		\begin{equation}\label{3.6}\begin{array}{l}
				\left\{
				\begin{array}{l}
					-\triangle V^+ -f'(x,u(x))V^+\geq 0~~\mbox{in}~\Omega_n^+,\\
					V^+=0~~\mbox{on}~\partial\Omega_n^+,
				\end{array}
				\right.
		\end{array}\end{equation}
		and
		\begin{equation}\label{3.7}\begin{array}{l}
				\left\{
				\begin{array}{l}
					-\triangle V^--f'(x,u(x))V^-\geq 0~~\mbox{in}~\Omega_n^-,\\
					V^-=0~~\mbox{on}~\partial\Omega_n^-.
				\end{array}
				\right.
		\end{array}\end{equation}
		
       Next we divide the following proof into two cases.
		
		(i) Case 1: If $V^+$ and $V^-$ are both nonnegative in $\Omega_n^+$ and $\Omega_n^-$ respectively, by the definitions of $\Omega_n^+$ and $\Omega_n^-$, we have $V^+=V^-\equiv 0,$ then $u$ is symmetric with respect to the hyperplane $T_n$.
		
		(ii) Case 2: Without of loss generality, suppose by contradiction that one of the two functions, say $V^-(x)<0$ in somewhere of $\Omega_n^-$. Now we consider a connected component $D$ in $\Omega_n^-$ such that $V^-(x)<0$ in $D$ and multiply (\ref{3.7}) by $V^-(x)$. By integrating, we obtain
		$$\int_D|\nabla V^-(x)|^2dx-\int_Df'(x,u(x))(V^-(x))^2dx\leq 0,$$
		which means that $\lambda_1(\mathcal{L},D)\leq 0$. By the monotonicity of first eigenvalue with respect to domains, then $\lambda_1(\mathcal{L},\Omega_n^-)\leq0,$ which contradicts the fact $\lambda_1(\mathcal{L},\Omega_n^{-})>0.$ Thus $V^-(x)\geq 0$ in $\Omega_n^-$, similarly, $V^+(x)\geq 0$ in $\Omega_n^+$. Therefore $u$ is symmetric with respect to the hyperplane $T_n$, namely, $u(x_1,x_2,\cdots,-x_n)=u(x_1,x_2,\cdots,x_n)$. This completes the proof of Lemma \ref{le3.2}. \end{proof}

	~~~In the rest of this section, we will show the proof of Theorem \ref{thm1.1}.
	
	\begin{proof}[Proof of Theorem \ref{thm1.1}]
		Definition 3 means that $\Omega$ is rotationally symmetric with respect to $x_n$ axis. Now we suppose that $T_{x_n}$ is an any hyperplane passing through $x_n$ axis. By the symmetric assumption of $f(x,u)$ with respect to the $x$-variable and Lemma \ref{le3.2}, we know that $u$ is symmetric with respect to the hyperplane $T_{x_n}$, which implies that $u$ is rotationally symmetric with respect to $x_n$ axis. It means that $u$ is radial when $\Omega$ is a ball. If the positive solution $u$ satisfies $\frac{\partial u}{\partial r}<0$ in initial  domain $\Omega_0=B_a(0)$, where $B_a(0)$ is a ball with radius $a$. It implies that the positive solution $u$ of (\ref{1.3}) is a Morse function and has exactly one critical point.  Next we need to prove that the positive solution $u$ satisfies $\frac{\partial u}{\partial r}<0$ in a ball $\Omega_0=B_a(0)$.
		
		Our proof is based mainly on the maximum principle. The idea of using continuity method is inspired by Jerison
		and Nadirashvili's method in \cite{Jerison}. We will divide the proof into four steps.
		
		{\bf Step 1. Radial monotonicity of the positive solution in a ball}
		
		Without loss of generality, we only prove the case on the positive $x_1$ semi-axis. Write $p=(x_1,y)\in \Omega_0$ for any $y\in \mathbb{R}^{n-1}.$ For $\lambda\in (0,a)$, we define
		\begin{equation*}\label{}\begin{array}{l}
				\left\{
				\begin{array}{l}
					\sum_\lambda=\{p=(x_1,y)\in\Omega_0: x_1>\lambda\},\\
					T_\lambda=\{p=(x_1,y)\in\Omega_0: x_1=\lambda\},\\
					\sum^\prime_\lambda=\mbox{the~reflection~domain~of} ~\sum_\lambda \mbox{with~ respect~ to}~T_\lambda ,\\
					p_\lambda=(2\lambda-x_1,y).
				\end{array}
				\right.
		\end{array}\end{equation*}
		
		In domain $\sum_\lambda$, we set
		\begin{equation*}\label{}\begin{array}{l}
				w_\lambda(p)=u(p)-u(p_\lambda)~\mbox{for~any}~p\in\sum_\lambda.
		\end{array}\end{equation*}
		
		Then we have by the mean value theorem
		\begin{equation*}\label{}\begin{array}{l}
				\left\{
				\begin{array}{l}
					\triangle w_\lambda(p)+c(p,\lambda)w_\lambda(p)=0~\mbox{in}~\sum_\lambda,\\
					w_\lambda<0~\mbox{on}~\partial\sum_\lambda \setminus T_\lambda,\\
					w_\lambda\equiv 0~\mbox{on}~\partial\sum_\lambda\cap T_\lambda.
				\end{array}
				\right.
		\end{array}\end{equation*}
		where $c(p,\lambda)$ is a bounded function in $\sum_\lambda$.

		Next we claim that
		\begin{equation}\label{3.8}\begin{array}{l}
				w_\lambda<0~\mbox{in}~\sum_\lambda~\mbox{for~any}~ \lambda\in(0,a).
		\end{array}\end{equation}
		This implies in particular that $w_\lambda$ attains its maximum in $\overline{\sum_\lambda}$ along $\partial\sum_\lambda\cap T_\lambda$. By Hopf Lemma, we have
		$$D_{x_1}w_\lambda|_{x_1=\lambda}=2D_{x_1}u|_{x_1=\lambda}<0,$$
		for any $\lambda\in(0,a)$.
		
		For any $\lambda$ close to $a$, by Proposition 2.13 (the maximum principle for a narrow domain) \cite{Han}, we have $w_\lambda<0.$ Suppose that $(\lambda_0,a)$ is the largest interval of values of $\lambda$ such that $w_\lambda<0$ in $\sum_\lambda$. Next we want to prove that $\lambda_0=0.$ If $\lambda_0>0$, by continuity, we have $w_{\lambda_0}\leq 0$ in $\sum_{\lambda_0}$ and $w_{\lambda_0}\not\equiv 0$ on $\partial\sum_{\lambda_0}$. Serrin's comparison principle implies that $w_{\lambda_0}< 0$ in $\sum_{\lambda_0}$.
		Suppose by contradiction that $w_{\lambda_0}< 0$ in $\sum_{\lambda_0}$ for $\lambda_0>0$. We will prove that
		\begin{equation*}\label{}\begin{array}{l}
				w_{\lambda_0-\varepsilon}< 0 ~\mbox{in}~\sum_{\lambda_0-\varepsilon},
		\end{array}\end{equation*}
		for any small $\varepsilon>0$.
		
		Fix $\sigma>0$ (the following will be determined). Let $\mathbb{W}$ be a closed subset in $\sum_{\lambda_0}$ such that $|\sum_{\lambda_0}\backslash\mathbb{W}|<\frac{\sigma}{2}$. Then $w_{\lambda_0}< 0$ in $\sum_{\lambda_0}$ implies that
		
		\begin{equation*}\label{}\begin{array}{l}
				w_{\lambda_0}(p)\leq -\alpha<0~\mbox{for~any}~p\in\mathbb{W}~\mbox{and~some}~\alpha>0.
		\end{array}\end{equation*}
		According to the continuity, we get
		\begin{equation}\label{3.9}\begin{array}{l}
				w_{\lambda_0-\varepsilon}< 0 ~\mbox{in}~\mathbb{W}.
		\end{array}\end{equation}
		We choose $\varepsilon>0$ to be small enough that$|\sum_{\lambda_0-\varepsilon}\backslash\mathbb{W}|<\sigma$ holds. We choose $\sigma$ in such a way that we can apply  Theorem 2.32 (the maximum principle for a domain with small volume) \cite{Han} to $w_{\lambda_0-\varepsilon}$ in $\sum_{\lambda_0-\varepsilon}\backslash\mathbb{W}$. Then we have that
		\begin{equation*}\label{}\begin{array}{l}
				w_{\lambda_0-\varepsilon}\leq 0~\mbox{in}~ \sum_{\lambda_0-\varepsilon}\backslash\mathbb{W}.
		\end{array}\end{equation*}
		In addition, by Serrin's comparison principle \cite{Serrin}, we have that
		\begin{equation}\label{3.10}\begin{array}{l}
				w_{\lambda_0-\varepsilon}< 0~\mbox{in}~ \sum_{\lambda_0-\varepsilon}\backslash\mathbb{W}.
		\end{array}\end{equation}
		By (\ref{3.9}) and (\ref{3.10}), we get a contradiction, therefore
		\begin{equation*}\label{}\begin{array}{l}
				w_\lambda<0~\mbox{in}~\sum_\lambda~\mbox{for~any}~ \lambda\in(0,a).
		\end{array}\end{equation*}
		
		By $w_\lambda\equiv 0$ on $\partial\sum_\lambda\cap T_\lambda$ and Hopf Lemma, we obtain that
		\begin{equation*}\label{}\begin{array}{l}
				D_{x_1}w_\lambda|_{x_1=\lambda}=2D_{x_1}u|_{x_1=\lambda}<0 ~\mbox{for~any}~\lambda\in(0,a).
		\end{array}\end{equation*}
		Similarly, we also can prove the case on the negative $x_1$ semi-axis,  other positive and negative $x_i(i=2,3,\cdots,n)$ semi-axis respectively. By $\partial_{x_i}u(r)=\partial_{r}u(r)\frac{x_i}{r}$, and the sum of $x_i\partial_{x_i}u$ with respect to $i=1,2,\cdots, n$, we have that
		\begin{equation}\label{3.11}
			\frac{\partial u(r)}{\partial r}=\frac{1}{r}\sum_i^{n}x_i\partial_{x_i}u<0~\mbox{for}~r\neq 0.
		\end{equation}
		
		{\bf Step 2. Small perturbation}
		
		In the rest of this section, without further mentioning, we always assume that the domains are simple rotationally symmetric along $x_n$ axis.
		
		We consider the following family of domains with continuous variation
		\begin{equation}\label{3.12}
			\Omega_t:=\{x=(r,x_n): |x_n|<tg(r)+(1-t)\sqrt{a^2-r^2}, 0\leq r=(\sum_{i=1}^{n-1}x_i^2)^{\frac{1}{2}}<a\}, t\in [0, 1],
		\end{equation}
		where $\Omega_0=B_a$ with radius $a$, and $\Omega_1=\Omega$ is our target domain. By the assumption of stable solution and Lemma \ref{le3.1}, we know that the solution $u_t$ of (\ref{1.3}) in $\Omega_t$ is unique. Then the following two facts are well known:\\
		(1) $u_t$ is a continuous function of $t$ for continuously differentiable perturbation (see \cite{Courant,Kato});\\
		(2) By elliptic regularity this in turn gives that $u_t$ converges uniform to $u_{t_0}$ in sense of $C^k$ as $t \rightarrow t_0$ up to
		the boundary if boundary is $C^k$.
		
		By Lemma \ref{le2.4} and Lemma \ref{le3.2}, we know that $u_t$ is a Morse function in $\Omega_t$ for $t$ small enough, so $u_t$ has exactly one critical point in $\Omega_t$. In fact, we will show that $\partial_{x_n}u_t<0$ for $x_n>0$ in $\Omega_t$ for $t$ small. This can be done by the maximum principle. Without loss of generality, we let $v=\partial_{x_n}u_t$ in the subdomain $\Omega_{t,n}^+:=\{x\in \Omega_t:x_n>0\}.$  Now we claim that
		$$ v\leq 0~\mbox{on}~\partial\Omega_{t,n}^+.$$
		In fact, by Lemma \ref{le3.2}, we know that $u_t$ is an even function with respect to $x_n$, namely, $u(\cdot, x_n)=u(\cdot, -x_n)$, then $v=\partial_{x_n} u_t=0$ on the flat boundary $x_n=0$ of $\partial\Omega_{t,n}^+$.  In addition, since $t$ is small enough, by the continuity, we have that $v=\partial_{x_n}u_t\leq 0$ on the curved portion of $\partial\Omega_{t,n}^+$. Now we take the partial derivative of the equation in (\ref{1.3}) with respect to $x_n$, and by the monotonicity of $f$ in $x_n$-variable for $x_n>0$, then $v$ satisfies
		\begin{equation}\label{3.13}\begin{array}{l}
				\left\{
				\begin{array}{l}
					\triangle v+\frac{\partial f(x,u_t)}{\partial u_t}v=-\frac{\partial f}{\partial x_n}\geq 0~\mbox{in}~\Omega_{t,n}^+,\\
					v\leq 0 ~\mbox{on}~ \partial\Omega_{t,n}^+.
				\end{array}
				\right.
		\end{array}\end{equation}
		It follows from the assumption of stable solution and Lemma \ref{le2.2} that $v<0$ in $\Omega_{t,n}^+$, namely, $\partial_{x_n}u_t<0$ for $x_n>0$ in $\Omega_t$. By the symmetry of solution $u$, we can have that $\partial_{x_n}u_t>0$ for $x_n<0$ in $\Omega_t$.
		
		Next we will show that
		\begin{equation}\label{3.14}
			x_i\partial_{x_i} u_t<0,~\mbox{for any}~x_i\neq 0,i=1,2,\cdots, n-1.
		\end{equation}
		
		Without loss of generality, we let $w=\partial_{x_i} u_t$ in the subdomain $\Omega_{t,i}^+:=\{x\in\Omega_t : x_i>0\}.$ In order to
		apply the maximum principle to prove above mentioned monotonicity properties of $w$, we claim that
		$$w\leq 0~\mbox{on}~\partial\Omega_{t,i}^+.$$
		In fact, by the result of $u$ is rotationally symmetric with respect to $x_n$ axis, that is, $u_t(x)=u_t(r,x_n),$ where $r=|x'|, ~ x'=(x_1,x_2,\cdots,x_{n-1})$, we know that $u_t$ is an even function with respect to $x_i$ for any $i=1,2,\cdots, n-1$. Then $w=\partial_{x_i} u_t=0$ on the flat boundary $x_i=0$ of $\partial\Omega_{t,i}^+$.  In addition, because of $t$ small enough, by the continuity, we have that $\partial_{x_i}u_t\leq 0$ on the curved portion of $\partial\Omega_{t,i}^+$. Similarly, we take the partial derivative of the equation in (\ref{1.3}) with respect to $x_i$, then $w$ satisfies
		\begin{equation}\label{3.15}\begin{array}{l}
				\left\{
				\begin{array}{l}
					\triangle w+\frac{\partial f(x,u_t)}{\partial u_t}w=-\frac{\partial f}{\partial x_i}\geq 0~\mbox{in}~\Omega_{t,i}^+,\\
					w\leq 0 ~\mbox{on}~ \partial\Omega_{t,i}^+.
				\end{array}
				\right.
		\end{array}\end{equation}
		By the assumption of stable solution and Lemma \ref{le2.2},we know that $w<0$ in $\Omega_{t,i}^+$, so $x_i\partial_{x_i}u_t<0$ in $\Omega_{t,i}^+$. Similarly, we can obtain that $x_i\partial_{x_i}u_t<0$ for $x_i<0$ in $\Omega_{t}$.
		
		Finally, by $\partial_{x_i}u_t(r,x_n)=\partial_{r}u_t(r,x_n)\frac{x_i}{r}$, and the sum of $x_i\partial_{x_i}u_t$ with respect to $i=1,2,\cdots, n-1$, we have that
		\begin{equation}\label{3.16}
			\frac{\partial u_t(r, x_n)}{\partial r}=\frac{1}{r}\sum_i^{n-1}x_i\partial_{x_i}u_t<0~\mbox{for}~r\neq 0.
		\end{equation}
		This completes  Step 2 for $t$ small enough.
		
		{\bf Step 3. Exactly one critical point in $\Omega_t$ for the limiting case}
		
		We  set up the usual contradiction argument.  Suppose that there exists a $t$ ($0<t<1$) such
		that the conclusion of Theorem \ref{thm1.1} fails. Let $\eta$ be the infimum of such $t$. It follows from Step 2 that $\eta>0$. Then Theorem \ref{thm1.1} holds for $t<\eta$. Theorem \ref{thm1.1} no longer holds for $t=\eta$. By continuity, we have that
		\begin{equation*}\begin{array}{l}
				\left\{
				\begin{array}{l}
					(1) \partial_{x_n}u_\eta\leq 0~\mbox{for}~x_n>0,\\
					(2) x_i\partial_{x_i}u_\eta\leq 0~\mbox{for}~x_i\neq 0, i=1,2,\cdots,n-1,\\
					(3)\frac{\partial u_\eta(r, x_n)}{\partial r}\leq0,~\mbox{for}~ r\neq 0, ~\mbox{where}~ r=|x'|, x'=(x_1,x_2,\cdots,x_{n-1}).
				\end{array}
				\right.
		\end{array}\end{equation*}
		
		The strong maximum principle implies that strict inequalities hold in corresponding domain, that is
		\begin{equation}\label{3.17}\begin{array}{l}
				\left\{
				\begin{array}{l}
					\partial_{x_n}u_\eta< 0~\mbox{for}~x_n>0,\\
					x_i\partial_{x_i}u_\eta< 0~\mbox{for}~x_i\neq 0, i=1,2,\cdots,n-1.
				\end{array}
				\right.
		\end{array}\end{equation}
		
		Next we will show that $u_\eta$ has no critical point in $\Omega_\eta\setminus{\{o\}}.$ Assume that there is another critical point $p$ in $\Omega_\eta\setminus{\{o\}},$ then we have that
		$$\partial_{x_i}u_\eta(p)=0,~\mbox{for any}~i=1,2,\cdots,n.$$
		This contradicts (\ref{3.17}), then $u_\eta$ dos not has extra critical point in $\Omega_\eta\setminus{\{o\}}$.
		
		{\bf Step 4. Non-degeneracy of the critical point $o$ of $u_\eta$ in $\Omega_\eta$}
		
		In this step, without loss of generality, we consider subdomain $\Omega_{\eta,n,i}^+:=\{x\in \Omega_\eta:x_n>0,x_i>0\}, (i=1,2,\cdots,n-1).$ We take the partial derivative of the equation in (\ref{1.3}) with respect to $x_n$ and $x_i$ respectively, then $v=\partial_{x_n}u_\eta$ and $w=\partial_{x_i}u_\eta$ respectively satisfies
		$$\triangle v+\frac{\partial f(x,u_\eta)}{\partial u_\eta}v=-\frac{\partial f}{\partial x_n}\geq 0~\mbox{in}~\Omega_{\eta,n,i}^+,$$
		and
		$$\triangle w+\frac{\partial f(x,u_\eta)}{\partial u_\eta}w=-\frac{\partial f}{\partial x_i}\geq 0~\mbox{in}~\Omega_{\eta,n,i}^+.$$
		
		By the results of Step 2 and Step 3, we know that
		\begin{equation}\label{3.18}
			v=\partial_{x_n}u_\eta<0,~w=\partial_{x_i}u_\eta<0~\mbox{in}~ \Omega_{\eta,n,i}^+.
		\end{equation}
		
		Since $\partial_{x_n}u_\eta(o)=0, \partial_{x_i}u_\eta(o)=0,$ let us apply Serrin's lemma (Lemma \ref{le2.3}) at origin $o$ for $\partial_{x_n}u_\eta$ and $\partial_{x_i}u_\eta$ in $\Omega_{\eta,n,i}^+$, where $o$ on the sub-manifold $\{x_n=0\}\cap\{x_i=0\}$.
		Let us take $x_n$ direction as the direction $s$ for $\partial_{x_n}u_\eta$, $x_i$ direction as the direction $s$ for  $\partial_{x_i}u_\eta$. Thus we arrive at
		\begin{equation}\label{3.19}
			\partial_{x_nx_n}u_\eta<0,~\partial_{x_ix_i}u_\eta<0~\mbox{at}~ o.
		\end{equation}
		
		Since $u_\eta$ is rotationally symmetric with respect to $x_n$ axis, then the local Taylor expansion of $u_\eta$ at $o$ can be given by
		\begin{equation}\label{3.20}
			u_\eta(x)=u_\eta(r,x_n)=M+P_2+o(2),
		\end{equation}
		where $M:=u_\eta(o)$ is the maximum of $u_\eta$, and
		\begin{equation*}\label{}
			P_2=\sum_{i,j=1}^na_{ij}x_ix_j.
		\end{equation*}
		
		Since $u_\eta$ is even in $x_i$ for $i=1,2,\cdots,n-1$, thus the mixed terms $a_{ij}=0$ for $i\neq j, i=1,2,\cdots,n-1; j=1,2,\cdots,n$. Because of $u_\eta$ is rotationally symmetric with respect to $x_n$ axis, then
		$$2c_1:=a_{11}=a_{22}=\cdots=a_{n-1,n-1},~2c_2:=a_{nn}.$$
		
		It follows from (\ref{3.19}) that
		$$c_1<0, c_2<0.$$
		It means that
		$$\det|\partial_{x_ix_j}u_\eta(o)|\neq 0.$$
		Therefore $u_\eta$ is a Morse function in $\Omega_\eta.$ This completes the proof of Step 4.
		
		Up to now, we have verified all the conclusions of Theorem \ref{thm1.1} for $\Omega_\eta$. The end of Step 2 implies
		that the monotone properties of $u_\eta$ can be preserved a bit further for $t>\eta$. This contradicts to the definition of
		$\eta$, which completes the proof of Theorem \ref{thm1.1}.
	\end{proof}
	
	According to Step 2, we make the following remark.
	\begin{Remark}\label{rem3.3}
		Our initial domain $\Omega_0$ is a ball can be relaxed. In fact, by the proof of Step 2, we note that for simple rotationally symmetric domain as soon as i) the initial domain satisfied desired monotonicity properties as stated in Theorem \ref{thm1.1}; ii) the function is a Morse function, which has exactly one critical point $o$, $o$ is the origin of coordinates. Therefore the anticipating properties state in Theorem \ref{thm1.1} preserved under small deformation of such domains. In addition, the related results can be extended to more general uniformly elliptic equation, e.g., quasilinear uniformly elliptic equation.
	\end{Remark}

\section{Proof of Theorem \ref{thm1.2} }
~~~~In this section, similar to Lemma \ref{le3.2}, we use the symbols $T_n,\Omega^-_n$ and $\Omega^+_n$ in (\ref{3.1}).
Suppose that $u$ is a solution of (\ref{1.4}) and denote by $\mathcal{L}$ the linearized operator at $u$, namely
$$\mathcal{L}=-\triangle -f'(u).$$
In addition, let $\lambda_1(\mathcal{L},D)$ be the first eigenvalue of $\mathcal{L}$ in a subdomain $D\subset\Omega$ with zero Dirichlet boundary condition. Then we have the following lemmas.
	
	\begin{Lemma}\label{le4.1}
		Let $u$ be a stable solution of (\ref{1.4}). If $f(u)$ is convex with respect to $u$, then $u$ is symmetric with respect to the hyperplane $T_n$, i.e. $u(x_1,x_2,\cdots,-x_n)=u(x_1,x_2,\cdots,x_n)$.
	\end{Lemma}
	\begin{proof} Firstly, we assume that $v^+$ and $v^-$ are the reflected functions of $u$ in the domains $\Omega_n^+$ and $\Omega_n^-$, respectively,
		\begin{equation*}\label{}
			\left\{
			\begin{array}{l}
				v^+(x)=u(x_1,x_2,\cdots,-x_n) ~~\mbox{for}~x\in\Omega_n^+,\\
				v^-(x)=u(x_1,x_2,\cdots,-x_n) ~~\mbox{for}~x\in\Omega_n^-.
			\end{array}
			\right.
		\end{equation*}
		Since $f(u)$ is convex, we have
		\begin{equation*}\label{}
			\left\{
			\begin{array}{l}
				f(v^+(x))-f(u(x))\geq f'(u(x))(v^+(x)-u(x)) ~~\mbox{in}~x\in\Omega_n^+,\\
				f(v^-(x))-f(u(x))\geq f'(u(x))(v^-(x)-u(x))~~\mbox{in}~x\in\Omega_n^-.
			\end{array}
			\right.
		\end{equation*}
		 According to (\ref{1.4}), denote by the functions $w^+=v^+-u$ and $w^-=v^--u$, then we have
		\begin{equation}\label{4.2}\begin{array}{l}
				\left\{
				\begin{array}{l}
					-\triangle w^+ -f'(u(x))w^+\geq 0~~\mbox{in}~\Omega_n^+,\\
					w^+=0~~\mbox{on}~\partial\Omega_n^+,
				\end{array}
				\right.
		\end{array}\end{equation}
		and
		\begin{equation}\label{4.3}\begin{array}{l}
				\left\{
				\begin{array}{l}
					-\triangle w^--f'(u(x))w^-\geq 0~~\mbox{in}~\Omega_n^-,\\
					w^-=0~~\mbox{on}~\partial\Omega_n^-,
				\end{array}
				\right.
		\end{array}\end{equation}
		 The following proof is similar to Lemma \ref{le3.2}, we omit the proof. \end{proof}
	
	In the rest of this section, we will present the proof of Theorem \ref{thm1.2}.
	
	\begin{proof}[Proof of Theorem \ref{thm1.2}]
		Definition 2 means that $\Omega$ is rotationally symmetric with respect to $x_n$ axis. Now we suppose that $T_{x_n}$ is an any hyperplane passing through $x_n$ axis. By the symmetric assumption of domain $\Omega$ and Lemma \ref{le4.1}, we know that $u$ is symmetric with respect to the hyperplane $T_{x_n}$. This implies that $u$ is rotationally symmetric with respect to $x_n$ axis. It means that $u$ is radial when $\Omega$ is a ball.  In the case of domains deformation, the remaining proof is similar to Theorem \ref{thm1.1}. Next we will divide the proof into three steps.
		
		{\bf Step 1. Small perturbation}
		
		Domains $\Omega_t$ is defined in (\ref{3.13}), where $\Omega_0=B_a$ with radius $a$ and $\Omega_1=\Omega$ is our target domain. By the assumption of stable solution and Lemma \ref{le3.1}, we know that the solution $u_t$ to (\ref{1.4}) in $\Omega_t$ is unique. Then the following two facts are well known:\\
		(1) $u_t$ is a continuous function of $t$ for continuously differentiable perturbation (see \cite{Courant,Kato});\\
		(2) By elliptic regularity this in turn gives that $u_t$ converges uniform to $u_{t_0}$ in sense of $C^k$ as $t \rightarrow t_0$ up to
		the boundary if boundary is $C^k$.
		
		By Lemma \ref{le2.4}, and Lemma \ref{le4.1}, we know that $u_t$ is a Morse function in $\Omega_t$ for $t$ small enough, so $u_t$ has exactly one critical point in $\Omega_t$. In fact, similar to the step 2 of Theorem \ref{thm1.1}, we can obtain
		
		\begin{equation}\label{4.4}\begin{array}{l}
				\left\{
				\begin{array}{l}
					\partial_{x_n}u_t<0~~\mbox{for }~x_n>0,\\
					x_i\partial_{x_i} u_t<0,~\mbox{for any}~x_i\neq 0,i=1,2,\cdots, n-1,
				\end{array}
				\right.
		\end{array}\end{equation}
		for $t$ small enough.
		
		{\bf Step 2. Exactly one critical point in $\Omega_t$ for the limiting case}
		
		We  set up the usual contradiction argument.  Suppose that there exists a $t$ ($0<t<1$) such
		that the conclusion of Theorem \ref{thm1.2} fails. Let $\eta$ be the infimum of such $t$. It follows from Step 1 that $\eta>0$. Then Theorem \ref{thm1.2} holds for $t<\eta$. Theorem \ref{thm1.2} no longer holds for $t=\eta$. By continuity, we have that
		\begin{equation*}\begin{array}{l}
				\left\{
				\begin{array}{l}
					(1) \partial_{x_n}u_\eta\leq 0~\mbox{for}~x_n>0,\\
					(2) x_i\partial_{x_i}u_\eta\leq 0~\mbox{for}~x_i\neq 0, i=1,2,\cdots,n-1,\\
					(3)\frac{\partial u_\eta(r, x_n)}{\partial r}\leq0,~\mbox{for}~ r\neq 0, ~\mbox{where}~ r=|x'|, x'=(x_1,x_2,\cdots,x_{n-1}).
				\end{array}
				\right.
		\end{array}\end{equation*}
		
		The strong maximum principle implies that strict inequalities hold in corresponding domain, that is
		\begin{equation}\label{4.5}\begin{array}{l}
				\left\{
				\begin{array}{l}
					\partial_{x_n}u_\eta< 0~\mbox{for}~x_n>0,\\
					x_i\partial_{x_i}u_\eta< 0~\mbox{for}~x_i\neq 0, i=1,2,\cdots,n-1.
				\end{array}
				\right.
		\end{array}\end{equation}
		
		Next we will show that $u_\eta$ has no critical point in $\Omega_\eta\setminus{\{o\}}.$ Assume that there is another critical point $p$ in $\Omega_\eta\setminus{\{o\}},$ then we have that
		$$\partial_{x_i}u_\eta(p)=0,~\mbox{for any}~i=1,2,\cdots,n.$$
		This contradicts (\ref{4.5}), then $u_\eta$ has no extra critical point in $\Omega_\eta\setminus{\{o\}}$.
		
		{\bf Step 3. Non-degeneracy of the critical point $o$ of $u_\eta$ in $\Omega_\eta$}
		
		In this step, without loss of generality, we consider subdomain $\Omega_{\eta,n,i}^+:=\{x\in \Omega_\eta:x_n>0,x_i>0\}, (i=1,2,\cdots,n-1).$ We take the partial derivative of the equation in (\ref{1.4}) with respect to $x_n$ and $x_i$ respectively, then $v=\partial_{x_n}u_\eta$ and $w=\partial_{x_i}u_\eta$ respectively satisfies
		$$\triangle v+f'(u_\eta)v=0~\mbox{in}~\Omega_{\eta,n,i}^+,$$
		and
		$$\triangle w+f'(u_\eta)w=0~\mbox{in}~\Omega_{\eta,n,i}^+.$$
		
		By the results of Step 1 and Step 2, we know that
		
		\begin{equation}\label{4.6}
			\partial_{x_n}u_\eta<0,~\partial_{x_i}u_\eta<0~\mbox{in}~ \Omega_{\eta,n,i}^+.
		\end{equation}
		The rest of the proof is similar to the step 4 of Theorem \ref{thm1.1}. Here, we omit the proof.
		
		Up to now, we have verified all the conclusions of Theorem \ref{thm1.2} for $\Omega_\eta$. The end of Step 1 implies
		that the monotone properties of $u_\eta$ can be preserved a bit further for $t>\eta$. This contradicts to the definition of
		$\eta$, which completes the proof of Theorem \ref{thm1.2}.
	\end{proof}
	
	\begin{Remark}\label{rem4.3}
		For a planar bounded simple rotationally symmetric domain $\Omega,$ i.e. $\Omega$ is a planar Steiner symmetric domain, by continuity method and a variety of maximum principles, we can also prove that a positive stable solution $u$ of (\ref{1.1}) has a unique critical point, and this critical point is non-degenerate, where $f(x,u)$ is a convex function with respect to $u$ and decreasing in the $x_i$-variable for $x_i>0(i=1,2)$.
	\end{Remark}

\noindent {\bf Conflict of Interest}  { The authors declare no conflict of interest.}


\begin{thebibliography}{99}\addtolength{\itemsep}{-1.5ex}\footnotesize
		\bibitem{Alberti} G.S. Alberti, G. Bal, M. Di Cristo, Critical points for elliptic equations with prescribed boundary conditions, Arch. Ration. Mech. Anal. 226 (2017), no. 1, 117-141.
		\bibitem{AlessandriniMagnanini1} G. Alessandrini, R. Magnanini, The index of isolated critical points and solutions of elliptic equations in the plane, Ann. Scuola Norm. Sup. Pisa Cl. Sci. 19 (4) (1992) 567-589.
		\bibitem{Anosov} D.V. Anosov, V.I. Arnold, Dynamical System I, Encyclopaedia of Mathematical Sciences, vol. 1, Spring-Verlag, 1985,
		pp. 80-92.
		\bibitem{Berestycki} H. Berestycki, L. Nirenberg, On the method of moving plane and the sliding method, Bol. Soc. Bras. Mat. Nova Ser. 22
		(1991) 1-37.
		\bibitem{BerestyckiVaradhan} H. Berestycki, L. Nirenberg, S.R.S. Varadhan, The principal eigenvalue and maximum principle for second order elliptic
		operators in general domains, Commun. Pure Appl. Math. 47 (1994) 47-92.
		\bibitem{CabreChanillo} X. Cabr\'{e}, S. Chanillo, Stable solutions of semilinear elliptic problems in convex domains, Selecta Math. (N.S.) 4 (1998) 1-10.
		\bibitem{Cheeger} J. Cheeger, A. Naber, D. Valtorta, Critical sets of elliptic equations, Comm. Pure Appl. Math. 68 (2015), no. 2, 173-209.
		\bibitem{Chen} H.B. Chen, Y. Li, L.H. Wang, Monotone properties of the eigenfunction of Neumann problems, J. Math. Pures Appl. (9) 130 (2019), 112-129.
		\bibitem{Courant} R. Courant, D. Hilbert, Methods of Mathematical Physics, John Wiley, New York, 1989.
		\bibitem{De} F. De Regibus, M. Grossi, On the number of critical points of the second eigenfunction of the Laplacian in convex planar domains, J. Funct. Anal. 283 (2022), no. 1, Paper No. 109496, 22 pp.
		\bibitem{Deng2019} H.Y. Deng, H.R. Liu, L. Tian, Critical points of solutions for the mean curvature equation in strictly convex and nonconvex domains, Israel J. Math. 233 (2019), no. 1, 311-333.
		\bibitem{Deng2018} H.Y. Deng, H.R. Liu, L. Tian, Critical points of solutions to a quasilinear elliptic equation with nonhomogeneous Dirichlet boundary conditions, J. Differential Equations 265 (2018), no. 9, 4133-4157.
		\bibitem{Deng2022} H.Y. Deng, H.R. Liu, X.P. Yang, Critical points of solutions to a kind of linear elliptic equations in multiply connected domains, Israel J. Math. 249 (2022), no. 2, 935-971.
		\bibitem{Deng2023} H.Y. Deng, H.R. Liu, X.P. Yang, On the number and geometric location of critical points of solutions to a semilinear elliptic equation in annular domains, arXiv:2310.02089 (2023).
		\bibitem{Deng2020} H.Y. Deng, H.R. Liu, L. Tian, Classification of singular set of solutions to elliptic equations, Commun. Pure Appl. Anal. 19 (2020), no. 6, 2949-2964.
		\bibitem{GidasNi} B. Gidas, W.M. Ni, L. Nirenberg, Symmetry and related properties via the maximum principle, Commun. Math. Phys.
		68 (1979) 209-243.
		\bibitem{GrossiLuo} M. Grossi, P. Luo, Critical points of positive solutions of nonlinear elliptic equations: multiplicity, location and non-degeneracy, Indiana Univ. Math. J. 72 (2023), no. 2, 821-871.
		\bibitem{GrossiIanni} M. Grossi, I. Ianni, P. Luo, S.S. Yan, Non-degeneracy and local uniqueness of positive solutions to the Lane-Emden problem in dimension two, J. Math. Pures Appl. (9) 157 (2022), 145-210.
		\bibitem{Han} Q. Han, F.H. Lin, Elliptic partial differential equations, Second edition. Courant Lecture Notes in Mathematics, 1. Courant Institute of Mathematical Sciences, New York; American Mathematical Society.
		\bibitem{Han1998} Q. Han, R. Hardt, F.H. Lin, Geometric measure of singular sets of elliptic equations, Comm. Pure Appl. Math. 51 (1998), no. 11-12, 1425-1443.
		\bibitem{Han2003} Q. Han, R. Hardt, F.H. Lin, Singular sets of higher order elliptic equations, Comm. Partial Differential Equations 28 (2003), no. 11-12, 2045-2063.
		\bibitem{Hardt1} R. Hardt, M. Hoffmann-Ostenhof, T. Hoffmann-Ostenhof, N. Nadirashvili, Critical sets of solutions to elliptic equations. J. Differential Geom. 51 (1999), no. 2, 359-373.
		\bibitem{Hoffmann-Ostenhof} M. Hoffmann-Ostenhof, T. Hoffmann-Ostenhof, N. Nadirashvili, Critical sets of smooth solutions to elliptic equations in dimension 3, Indiana Univ. Math. J. 45 (1996), no. 1, 15-37.
		\bibitem{Jerison} D. Jerison, N. Nadirashvili, The ``hot spots'' conjecture for domains with two axes of symmetry, J. Am. Math. Soc. 13 (4)
		(2000) 741-772.
		\bibitem{Kato} T. Kato, Perturbation Theory for Linear Operators, Classics in Mathematics, Springer, 2013.
		\bibitem{Kawohl} B. Kawohl, Rearrangement and Convexity of Level Set in PDE, Lecture Notes in Mathematics, vol. 1150, 1985.
		\bibitem{Lin2024} F.H. Lin, Z.W. Shen, Critical sets of solutions of elliptic equations in periodic homogenization, Comm. Pure Appl. Math. 77 (2024), no. 7, 3143-3183.
		\bibitem{Magnanini}  R. Magnanini, An introduction to the study of critical points of solutions of elliptic and parabolic equations, Rend. Istit. Mat. Univ. Trieste 48 (2016), 121-166.
		\bibitem{Milnor} J. Milnor, Morse Theory, Based on Lecture Notes by M. Spivak and R. Wells, Annals of Mathematics Studies, vol. 51,
		Princeton University Press, 1963.
		\bibitem{Naber} A. Naber, D. Valtorta, Volume estimates on the critical sets of solutions to elliptic PDEs, Comm. Pure Appl. Math. 70 (2017), no. 10, 1835-1897.
\bibitem{Pacella} F. Pacella, Symmetry results for solutions of semilinear elliptic equations with convex nonlinearities, J. Funct. Anal. 192 (2002), no. 1, 271-282.
		\bibitem{Serrin} J. Serrin, A symmetry problem in potential theory, Arch. Ration. Mech. Anal. 43 (1971) 304-318.
		
		
		
		
	\end{thebibliography}
\end{document}